\documentclass[final]{amsart}
\usepackage[utf8]{inputenc}
\usepackage[T1]{fontenc}
\usepackage{graphicx}
\usepackage{grffile}
\usepackage{longtable}
\usepackage{wrapfig}
\usepackage{rotating}
\usepackage[normalem]{ulem}
\usepackage{amsmath}
\usepackage{textcomp}
\usepackage{amssymb}
\usepackage{capt-of}
\usepackage[colorlinks=true,linkcolor=black,citecolor=black,urlcolor=blue]{hyperref}

\newcommand{\reg}{\operatorname{reg}}
\newcommand{\der}{\operatorname{der}}
\newcommand{\sreg}{\operatorname{\Sigma-reg}}

\newcommand{\git}{/\!\!/}
\newcommand{\Ch}{\mathrm{Ch}}
\newcommand{\tm}{\mathrm{tame}}
\usepackage{amsmath}
\usepackage{amssymb}
\usepackage{amsthm}
\usepackage{amscd}
\usepackage{mathtools}
\usepackage{marginnote}
\usepackage{todonotes}
\usepackage{bbm}
\usepackage{bm}
\usepackage[notcite,notref]{showkeys}
\usepackage{cite}
\usepackage[utf8]{inputenc}
\usepackage{tikz-cd}
\usepackage[colorlinks=true,linkcolor=black,citecolor=black,urlcolor=blue]{hyperref} 
\newtheorem{theorem}{Theorem}[section]
\newtheorem*{theorem*}{Theorem}

\newtheorem*{conjecture*}{Conjecture}

\newtheorem*{question*}{Question}

\newtheorem*{guess*}{Guess}

\newtheorem*{problem*}{Problem}

\newtheorem{lemma}[theorem]{Lemma}
\newtheorem*{lemma*}{Lemma}

\newtheorem*{exercise*}{Exercise}
\newtheorem{proposition}[theorem]{Proposition}
\newtheorem*{proposition*}{Proposition}

\newtheorem{corollary}[theorem]{Corollary}
\newtheorem*{corollary*}{Corollary}

\theoremstyle{definition}

\newtheorem{definition}[theorem]{Definition}
\newtheorem*{definition*}{Definition}
\newtheorem{remark}[theorem]{Remark}

\newtheorem*{example*}{Example}
\newtheorem*{examples*}{Examples}

\newcommand{\twomat}[4]{\begin{pmatrix} #1 & #2 \\ #3 & #4 \end{pmatrix}}

\newcommand{\pres}[1]{\left\langle #1 \right\rangle}

\renewcommand{\bar}{\overline}
\usepackage{mathtools}

\renewcommand{\AA}{\mathbb{A}}

\newcommand{\FF}{\mathbb{F}}
\newcommand{\F}{\mathbb{F}}
\newcommand{\GG}{\mathbb{G}}

\newcommand{\Q}{\mathbb{Q}}

\newcommand{\ZZ}{\mathbb{Z}}
\newcommand{\Z}{\mathbb{Z}}

\newcommand{\Cc}{\mathcal{C}}

\newcommand{\Oc}{\mathcal{O}}

\newcommand{\Tc}{\mathcal{T}}

\newcommand{\If}{\mathfrak{I}}

\newcommand{\Xf}{\mathfrak{X}}

\newcommand{\Zf}{\mathfrak{Z}}

\newcommand{\into}{\hookrightarrow}
\newcommand{\isomto}{\xrightarrow{\sim}}

\newcommand{\Aut}{\operatorname{Aut}}

\newcommand{\Lie}{\operatorname{Lie}}

\newcommand{\Gal}{\operatorname{Gal}}

\newcommand{\Ind}{\operatorname{Ind}}

\newcommand{\Spec}{\operatorname{Spec}}

\newcommand{\ord}{\operatorname{ord}}

\newcommand{\ch}{\operatorname{char}}
\newcommand{\Fr}{\operatorname{Fr}}

\newcommand{\ad}{\operatorname{ad}}

\newcommand{\im}{\operatorname{im}}

\newcommand{\LG}{{}^LG}
\newcommand{\GSp}{\mathrm{GSp}}

\DeclareMathOperator{\Ad}{Ad}

\author{Jack Shotton}
\date{\today}
\title{Irreducible components of the moduli space of Langlands parameters}
\hypersetup{
 pdfauthor={Jack Shotton},
 pdftitle={},
 pdfkeywords={},
 pdfsubject={},
 pdflang={English}}
\begin{document}

\begin{abstract}
    Let $F/\mathbb{Q}_p$ be finite, $G$ be an $L$-group, and let $\mathfrak{X}_G$ be the moduli space of Langlands
    parameters $W_F \to G$, in characteristic distinct from $p$. First, we determine the irreducible components of
    $\Xf_G$. Then, we determine the local structure around tamely ramified points for which the image of inertia
    is regular. This local structure is related to the endomorphism rings of Gelfand--Graev representations, by work of
    Li. Lastly, we determine an open dense set in $\Xf_M$, when $M$ is a Levi subgroup of $G$, such that the natural map
    of moduli stacks $[\Xf_M/M^\circ] \to [\Xf_G/G^\circ]$ is smooth on this set.
\end{abstract}

\maketitle
\section{Introduction}
\label{sec:intro}

Let $p$ be a prime and let $F$ be a finite extension of $\Q_p$ with Weil group $W_F$. Let \[G = G^\circ \rtimes \Gamma\]
with $G^\circ$ a split reductive group scheme\footnote{Recall that these are defined to have connected fibres.} over
$\Z[1/p]$ and $\Gamma$ a finite quotient of $W_F$ equipped with an action on $G^\circ$ by automorphisms preserving a
split Borel pair $(B,T)$.  Dat, Helm, Kurinczuk, and Moss \cite{datModuliLanglandsParameters2020} have shown that the
functor on $\Z[1/p]$-algebras
\[R \mapsto \{\text{$L$-homomorphisms $W_F^0 \to G(R)$}\}\] is representable by a scheme $\Xf_{G}$, locally of finite
type and flat over $\Z[1/p]$ --- here $W_F^0$ is a certain choice of discretization of $W_F$.\footnote{In
  \cite{datModuliLanglandsParameters2020}, they write $\LG$ for our $G$ and $\hat{G}$ for our $G^\circ$. This is more in
  keeping with the usual notation in the Langlands program, but leads to a proliferation of `$L$'s and hats that we
  prefer to avoid.}

For $L$ a field of characteristic zero, $\Xf_{G, L}$ is known to be generically reduced by work of
Bellovin-Gee \cite{bellovinGvaluedLocalDeformation2019} and so it has a smooth open dense subset (which may be made explicit). For
$L$ of characteristic $l$, this is far from being true, even for $G = GL_1$. However, one can still hope for a nice
description at an open dense set of points. For $G = GL_n$, in the tame case, we did this in
\cite{shottonGenericLocalDeformation2022}, finding a local description of $\Xf_{G}$ around a dense subset of its fibre
at each prime $l$. This description turns out to be related to the endomorphism algebra of the (integral) Gelfand-Graev
representation, and we applied this to the $l\neq p$ ``Breuil--M\'{e}zard'' conjecture.

Our aim here is to extend the geometric part of that work to general groups $G$. We do not quite succeed. Roughly, we
can only deal with components on which the $L$-parameter factors through a Levi subgroup in which `inertia is
regular'. For these components, we obtain a description similar to that for $GL_n$, again related to the endomorphism
algebras of Gelfand-Graev representations by work of Li \cite{liEndomorphismAlgebrasGelfandGraev2023}. For $G = GL_n$
this is everything, but for general $G$ there are components which fall outside our description. A limitation is that we
restrict to $G$ unramified; however, as in section~5.2 of \cite{datModuliLanglandsParameters2020}, in many situations
one may reduce to this case. (We also require $G^\circ$ to have smooth centre and simply connected derived subgroup).

The first task, of independent interest, is the determination of the geometrically irreducible components of $\Xf_{G}$
and its reductions modulo primes $l\ne p$. The (geometric) \emph{connected} components of $\Xf_{G}$, as well as of its
base changes to $\Z_l$ (equivalently, $\F_l$) for primes $l\ne p$, are determined in
\cite{datModuliLanglandsParameters2020}. The determination of the irreducible components takes a similar form, and was
explained to us some time ago by David Helm; we have written down the argument in section~\ref{sec:moduli} below.

We now outline the contents of this paper in more detail. First, the determination of irreducible components.  For $G$
an algebraic group acting on a scheme $X$, let $G_x$ denote the stabiliser of a point $x \in X(R)$ (it is a closed
subscheme of $G_R$).

\begin{theorem*}[Corollary~\ref{cor:irr-cpts}] Let $L$ be an algebraically closed field of characteristic distinct from
  $p$. There is a bijection between the set of irreducible components of $\Xf_{G, L}$ and the set of
  $G^\circ(L)$-conjugacy classes of pairs $(\eta, \alpha)$ where:
  \begin{itemize}
  \item $\eta$ is a continuous\footnote{For the discrete topology on $L$.} $L$-homomorphism $I^0_F \to G(L)$
    over $L$ that extends to $W^0_F$; and
  \item $\alpha$ is a component of the $(G^\circ)_\eta$-torsor of extensions of $\eta$ to $W^0_F$.
  \end{itemize}
\end{theorem*}
Note that two pairs $(\eta, \alpha)$ and $(\eta, \alpha')$ are $G^\circ(L)$-conjugate if and only if $\alpha$ and
$\alpha'$ are conjugate under $\pi_0((G^\circ)_\eta)$. 

The critical case (by the `tame reduction' of \cite{datModuliLanglandsParameters2020} sections~3 and~4.5 and the
discussion following Lemma~5.7) is when the action of $W_F$ on $G^\circ$ is unramified and $\eta$ is tamely ramified,
so determined by an element $\Sigma \in G^\circ(L)$. In this case the irreducible components of type $\Sigma$
are in bijection with twisted conjugacy classes of the component group of the centralizer of $\Sigma$ in $G^\circ$.

Assume, therefore, that $\Gamma$ corresponds to an unramified extension of $F$, and define $\Xf^{\tm}_{G}$ to be the moduli space of
\emph{tamely ramified} $L$-parameters. We have a presentation
\[W_F^0/P_F \cong \pres{\Fr, \sigma : \Fr \sigma \Fr^{-1} = \sigma^q} \] where $q$ is the order of the residue field of
$F$. Define $\Xf_{G}^{\sreg}$ to be the open subset of parameters $\rho$ where $\Sigma = \rho(\sigma) \in G^\circ$ is
regular. Let $T$ be a maximal split torus in $G^\circ$ preserved by the action of $W_F$. Let $T\git W$ be the GIT
quotient of $T$ by the action of the Weyl group $W$. There is a morphism, the Steinberg morphism,
\[\Ch : G^\circ \to G^\circ\git G^\circ  \isomto T\git W.\]
Let
\[(T\git W)^{\Fr^{-1}[q]}\] be the closed subscheme of $T\git W$ fixed by $\Fr^{-1}[q]$, where $[q]$ is the $q$th power
map. 
\begin{theorem*}[Theorem~\ref{thm:sreg}] Suppose that $G^\circ_{\der}$ is simply connected and that the centre
  of $G^{\circ}$ is smooth. Then the morphism
  \[\Xf_{G}^{\sreg} \to (T\git W)^{\Fr^{-1}[q]}\]
  sending $\rho$ to $\Ch(\rho(\sigma))$ is smooth and surjective.
\end{theorem*}

\begin{remark} \label{rmk:endomorphism} Suppose that $G$ is the $L$-group of an unramified group with smooth integral
  model $\GG/\Oc_F$ and that $G^\circ_{\der}$ is simply connected. In this case, the coordinate ring of
  $(T\git W)^{\Fr^{-1}[q]}$ has been shown in \cite{liEndomorphismAlgebrasGelfandGraev2023} in characteristic not
  dividing $|W|$, and in \cite{lishottonEndomorphismAlgebrasGelfand} in the case of good characteristic, to be
  isomorphic to the endomorphism ring of the Gelfand--Graev representation of $\GG(k_F)$. This is an inertial shadow of
  the (conjectural) local Langlands correspondence in families.
\end{remark}

\begin{remark}
  Under suitable assumptions on $G$ (as in \cite{xiaoVectorvaluedTwistedConjugate2019} Remark~4.3.4) we might expect
  a similar result when $G$ is tamely ramified. See Remark~\ref{rmk:ram} below.
\end{remark}
\begin{remark}
  In \cite{datModuliLanglandsParameters2020} Proposition~5.12 one finds conditions under which points of the special
  fibre with $\Sigma$ regular unipotent are smooth (at least for $l > h$, where $h$ is the Coxeter number of $G$). The
  condition takes the form $l \nmid \chi_G(q)$ where $\chi_G$ is a certain product of cyclotomic polynomials depending
  on $G$. Our result provides a clean description of $\Xf_G$ at such points when $l \mid \chi_G(q)$, as long as
  $G^\circ_{\der}$ is simply connected.
\end{remark}

Finally, we can go beyond the $\Sigma$-regular locus with the following result, for which we can drop the assumption
that $G$ is unramified. By a standard Levi subgroup of $\LG$ we mean a subgroup of the form $M = M^\circ \rtimes \Gamma$
where $M^\circ$ is a $\Gamma$-stable standard Levi subgroup of $G^\circ$ (relative to our fixed split maximal torus).
\begin{theorem*}[Theorem~\ref{thm:levi}]
  Suppose that $M$ is a standard Levi subgroup of $G$. Consider the morphism
  \[c: G^\circ \times\Xf_{M} \to \Xf_{G}\] sending $(g, \rho)$ to $g \rho g^{-1}$. Then there is an open subset
  $U \subset \Xf_{M}$ such that $c|_{U}$ is smooth and $U$ intersects each fibre of $\Xf_{M} \to \Spec\Z[1/p]$ in a
  dense open subset.
\end{theorem*}

\begin{remark}
  The theorem is in the spirit of Lemma~5.14 of \cite{datModuliLanglandsParameters2020}, which achieves a similar
  reduction for their study of unobstructedness.
\end{remark}

For $G = GL_n$, the image of \begin{equation}\label{eq:union-levi}\bigcup_{M} \im(G^\circ \times \Xf^{\sreg}_{M} \to
  \Xf^{\tm}_{G}),\end{equation} where $M$ runs over conjugacy classes of standard Levi subgroups, is fibrewise dense in
$\Xf_{G}$ (see \cite{shottonGenericLocalDeformation2022}). In general this is not true; there are two issues coming
from the disconnectedness of component groups and the existence of non-regular distinguished unipotent conjugacy
classes. We give examples in section~\ref{sec:examples}.

In \cite{shottonGenericLocalDeformation2022} we were able to use this and the desciption of the coordinate ring of
$(T\git W)^{[q]}$ as the endomorphism ring of a Gelfand--Graev representation\footnote{In fact, we just need the weaker
  result after $\otimes \bar{\Q}$.} to give a local proof of the author's ``$l \ne p$ Breuil--M\'{e}zard conjecture'' in
the tame case. It should be possible to do something similar for unramified groups, but only for modular $L$-parameters
which lie only on components in the image of~\eqref{eq:union-levi}. Given this serious restriction, we will not discuss this
point further here.

\subsection{Acknowledgments}
\label{sec:acknowledgments}

We thank David Helm for explaining to the author the results of section~\ref{sec:moduli}, and Tzu-Jan Li for helpful
conversations concerning the significance of $(T\git W)^{\Fr^{-1}[q]}$. We thank Sean Cotner for pointing out a gap in
the proof of Corollary 3.8, and helping to fix it.  Finally, we thank an anonymous referee for a thorough reading that
greatly improved this article.

\section{The moduli space of parameters}
\label{sec:moduli}

\subsection{Notation}\label{sec:notation} As in the introduction, let $F/\Q_p$ be a finite extension with ring of integers
$\Oc_F$ and residue field $k_F$ of order $q$. Let $W_F$ be the Weil group of $F$ with inertia group $I_F$ and wild
inertia group $P_F$. We choose a lift $\Fr \in W_F$ of arithmetic Frobenius and a lift $\sigma \in I_F$ of a topological
generator of $I_F/P_F$ in such a way that $\Fr \sigma \Fr^{-1} = \sigma^q$ (this can be done by
\cite{iwasawaGaloisGroupsLocal1955} Theorem 2).  Let $W_F^0$ be the
subgroup of $W_F$ generated by $\sigma, \Fr$, and $P_F$; its intersection with $I_F$ is denoted
$I_F^0$. We have that $I_F^0/P_F \cong \Z[1/p]$ via an isomorphism sending $\sigma$ to $1$. We topologise $W_F^0$ such that the
cosets of open subsets of $P_F$ (with its profinite topology) form a base for the topology on $W_F^0$.

Let $G$ be an algebraic group over $\Z[1/p]$ of the form $G^\circ \rtimes \Gamma$ where: $G^\circ$ is a split reductive
group scheme over $\Z[1/p]$, $\Gamma$ is a finite quotient of $W_F$, and $\Gamma$ acts on $G^\circ$ via automorphisms
preserving a split Borel pair $(B,T)$. We write the action of $W_F$ on $G^\circ$ using left superscripts:
$(w, g) \mapsto {}^wg$. Finally, we let $(X^*(T), \Delta, X_*(T), \Delta^\vee)$ be the root datum associated to
$G^\circ$, together with its action of $W_F$.

If $H$ is a subgroup of $W_F$ and $R$ is a $\Z[1/p]$-algebra, then an $L$-homomorphism $H \to G(R)$ is a continuous (for the
discrete topology on $G(R)$) homomorphism
\[\rho : H \to G(R)\]
such that the composite of $\rho$ with $G(R) \to \Gamma$ agrees with $H \to W_F \to \Gamma$ coming from the given
surjection. For $h \in H$, we write $\rho(h) = (\rho^\circ(h), h)$ (slightly loosely using the same letter $h$ for an
element of $H$ and its image in $\Gamma$). If $H\subset H'$ then by an extension of $\rho$ to $H'$ we will mean an
$L$-homomorphism $\tilde{\rho}: H' \to G(R)$ whose restriction to $H'$ is $\rho$. If $\gamma \in \Gamma$ then we will
write
\[G^\circ \rtimes \gamma = \{(g, \gamma) : g \in G^\circ\}\]
for the corresponding connected component of $G$.

When $G$ acts on a scheme $X$ and $\eta \in X(S)$, we write $G_\eta$ for the stabiliser group scheme (defined over
$S$). We adopt the convention that $G^\circ_\eta = (G^\circ)_\eta$, and not $(G_\eta)^\circ$.

If $X \to S$ and $T \to S$ are morphisms of schemes, then we will sometimes write $X_T = X \times_S T$ (particularly
when $T$ is a point or geometric point of $S$). If $A$ is a $\Z[1/p]$-algebra and $X$ is a scheme over $\Spec \Z[1/p]$
then we will write $X_A$ for $X \times_{\Spec \Z[1/p]} \Spec A$, \emph{reserving this notation for this situation}.

\subsection{The moduli space of $L$-parameters}
\label{sec:L-params}

One of the main results of \cite{datModuliLanglandsParameters2020} is the following.
\begin{theorem}[Dat--Helm--Kurinczuk--Moss] \label{thm:dhkm}The functor sending a $\Z[1/p]$-algebra $R$ to the set of
  $L$-homomorphisms $W_F^0 \to G(R)$ is representable by a scheme $\Xf_G$ over $\Z[1/p]$. Moreover, $\Xf_G$ has the
  following properties:
  \begin{enumerate}
  \item The morphism $\Xf_G \to \Spec \Z[1/p]$ is syntomic (flat, locally of finite presentation, with fibres that are
    local complete intersections) of relative dimension $\dim_{\Z[1/p]} G$, the relative dimension of $G$ over
    $\Z[1/p]$.
  \item If $L$ is a field of characteristic zero then $\Xf_G \times_{\Spec \Z[1/p]} \Spec L$ is reduced.\footnote{This
      part relies crucially on prior work of Bellovin-Gee \cite{bellovinGvaluedLocalDeformation2019}.}
  \item The connected components of $\Xf_G$ are finitely presented over $\Z[1/p]$.
  \end{enumerate}
\end{theorem}

If $\Lambda$ is a $\Z[1/p]$-algebra, then we write $\Xf_{G, \Lambda} = \Xf_{G} \times_{\Spec \Z[1/p]} \Spec
\Lambda$. Theorem~\ref{thm:dhkm} remains true with $\Z[1/p]$ replaced by $\Lambda$.

\subsection{Components}
\label{sec:connected}

Let $L$ be an algebraically closed field of characteristic $l$ distinct from $p$; we allow $l = 0$. We aim to
determine the irreducible components of $\Xf_G = \Xf_{G,L}$, after first recalling a result on the connected
components. 

\begin{theorem}[Dat--Helm--Kurinczuk--Moss]\label{thm:dhkm2} There are bijections between the sets of:
  \begin{enumerate}
  \item Connected components of $\Xf_{G, L}$; and
  \item $G^\circ(L)$-conjugacy classes of pairs $(\eta, \alpha)$ where $\eta : I_F \to G(L)$ is $G$-semisimple
    (see \cite{datModuliLanglandsParameters2020} Definition~4.12) and $\alpha$ is a connected component of the variety
    of extensions of $\eta$ to $W_F$.
  \end{enumerate}
  If $l > 0$ then these are in turn in bijection with the sets of connected components of $\Xf_{G, \bar{\F}_l}$ and of
  $\Xf_{G, W(\bar{\F}_l)}$.
\end{theorem}
\begin{proof} The bijection between the first two sets follows from \cite[Theorem 1.7]{datModuliLanglandsParameters2020}
  (see also \cite[Corollary~4.21]{datModuliLanglandsParameters2020} parts (2) and (5)).

  The final statement follows from \cite[Theorem~4.8]{datModuliLanglandsParameters2020}.
\end{proof}

\begin{remark}
  Since $\eta$ in (5) has finite image (since it is assumed to be continuous for the discrete topology on $L$), if $L$
  has characteristic zero then the $G$-semisimplicity condition is automatic.
\end{remark}

\emph{For the rest of this section, we base change everything to $L$ --- so $G = G_L$ etc.}

\begin{definition} An $L$-homomorphism $\eta : I_F^0 \to G(L)$ is \emph{admissible} if it extends to an $L$-homomorphism
  $W_F^0 \to G(L)$.
\end{definition}

Given an admissible $\eta$, we consider the affine $L$-scheme $Y_\eta$ defined on $L$-algebras $R$ by
\[Y_\eta(R) = \{\rho : W_F^0 \to G(R) \text{ extending $\eta$}\}.\] We write $\Sigma(\eta)$ for its set of connected
components (so we use the same notation as that for a similarly-defined set in \cite{datModuliLanglandsParameters2020}
section~3.2).

If we write $\Phi = \rho(\Fr)$ then
\[Y_\eta \cong \{\Phi \in G^\circ \rtimes \Fr : \Phi \eta(\gamma) \Phi^{-1} = \eta(\Fr\gamma\Fr^{-1}) \text{ for all
    $\gamma \in I_F^0$}\}\] and we see that $Y_\eta$ is a left $G^\circ_{\eta}$-torsor via left multiplication on
$\Phi$. If we fix $\rho_0$ extending $\eta$ with $\rho_0(\Fr) = \Phi_0$, then we can identify
$Y_\eta\cong G^\circ_{\eta}$, with a general point having $\rho(\Fr) = \phi\Phi_0$ for $\phi \in G^\circ_{\eta}$.

There is a second action of $G^\circ_{\eta}$ on $Y_\eta$ given by conjugation. On choosing $\rho_0$ as in the
previous paragraph, this may be identified with the action of $G^\circ_{\eta}$ on itself by $\Ad_{\Phi_0}$-twisted
conjugation.

If $\If$ is the moduli space of $L$-homomorphisms $I_F^0 \to G(R)$ (for $L$-algebras $R$), then $\If$ is a union of
connected components that are each affine of finite type over $L$. The $G^\circ$-orbit of $\eta \in \If(L)$,
$G^\circ \cdot \eta$, then naturally has the structure of a quasiaffine variety and, by
\cite{milneAlgebraicGroupsTheory2017} Corollary~7.13, there is an
isomorphism
\[G^\circ/G^\circ_\eta \isomto G^\circ \cdot \eta.\]

There is a natural morphism $p_I : \Xf_{G} \to \If$ and we write   
\[\Xf_{\eta} = p_{I}^{-1}(G^\circ\cdot \eta),\]
a locally closed subscheme of $\Xf_\eta$. Note that, by definition, we have
\[Y_\eta = p_I^{-1}(\eta).\]

Theorem~\ref{thm:irr-cpts-eta} and Corollary~\ref{cor:irr-cpts} below were explained to the author by David Helm.

\begin{theorem}\label{thm:irr-cpts-eta} For each admissible $L$-homomorphism $\eta : I_F^0 \to G(L)$, every connected
  component of $\Xf_\eta$ is irreducible, and its Zariski closure in $\Xf_G$ is an irreducible component of $\Xf_G$.

  If $\eta$ is chosen as above, then there is a bijection between:
  \begin{enumerate}
  \item The set of connected (or irreducible) components of $\Xf_\eta$;
  \item The set of $\pi_0(G_\eta^\circ)$-conjugacy classes in $\Sigma(\eta)$.
  \end{enumerate}
  If we choose a $\rho_0 : W^0_F \to G(L)$ extending $\eta$, then these are in bijection with the $\Ad_{\rho_0(\Fr)}$-twisted
  conjugacy classes in $\pi_0(G^\circ_\eta)$.
\end{theorem}
\begin{remark} From the proof of the theorem we will see that $|\Xf_{\eta}|$ is an open subset of an irreducible
  component of $|\Xf_{G}|$. However, if $\ch L \ne 0$, then neither $\Xf_{\eta}$ nor $\Xf_{G}$ need be reduced, and the
  non-reduced structures will typically not agree.
\end{remark}
\begin{proof} Note that we have an isomorphism
  \begin{align*}
Y_\eta \times G^\circ & \cong \Xf_\eta \times_{G^\circ\cdot \eta} G^\circ \\
(\rho, g)& \mapsto (g\rho g^{-1}, g)
    \end{align*} by definition. Taking the quotient by
  $G^\circ_{\eta}$ and using that $G^\circ/G^\circ_{\eta} \cong G^\circ \cdot \eta$, we obtain an
  isomorphism
  \[\Xf_\eta \cong G^\circ \times^{G^\circ_\eta} Y_\eta .\]
  It follows that $\Xf_\eta$ is a disjoint union of irreducible connected components corresponding to the
  $G^\circ_\eta$-orbits in $\pi_0(Y_\eta) = \Sigma(\eta)$, which are the same as the $\pi_0(G_\eta^\circ)$-orbits.

  Each component of $Y_\eta$ has dimension $\dim G^\circ_\eta$ and so the dimension of every component of $\Xf_\eta$ is
  \[\dim G^\circ + \dim G^\circ_\eta - \dim G^\circ_\eta = \dim G^\circ = \dim \Xf_{G}.\]
  Since $G^\circ\cdot \eta$ is locally closed, the same is true for $\Xf_\eta$. It follows that the underlying
  topological space of each connected component of $\Xf_\eta$ is an open subset of an irreducible component of
  $\Xf_\eta$.
\end{proof}
\begin{corollary}\label{cor:irr-cpts}
  There are bijections between:
  \begin{enumerate}
  \item The set of irreducible components of $\Xf_{G, L}$;
  \item The set of $G^\circ(L)$-conjugacy classes of pairs $(\eta, C)$ where $\eta : I_F^0 \to G(L)$ is an admissible
    $L$-homomorphism and $C$ is a connected component of $\Sigma(\eta)$.
  \end{enumerate}
  Let $\Tc$ be a set of representatives of the $G^\circ(L)$-conjugacy classes $\eta$ as in~(2). For each $\eta \in \Tc$,
  choose an extension $\rho_\eta$ of $\eta$ to $W_F^0$.
  
  Then these sets are in bijection with the set
  \[\{\text{$(\eta, C)\; :\; \eta \in \Tc$ and $C$ is an $\Ad_{\rho_\eta(\Fr)}$-twisted conjugacy class of
      $\pi_0(G^\circ_{\eta})$}\}.\]
\end{corollary}
\begin{proof}
  By Theorem~\ref{thm:irr-cpts-eta}, it is enough to prove that every irreducible component of $\Xf_G$
  arises as the closure of an irreducible component of some $\Xf_\eta$, for $\eta$ unique up to $G^\circ$-conjugacy.

  To this end, for each irreducible component $\Cc$ of $\Xf_{G}$, let $\rho$ be an $L$-point that lies only on that
  component and take $\eta = \rho|_{I_F^0}$. Then one of the components of $\Xf_\eta$ contains $\rho$, and its closure
  is an irreducible component of $\Xf_G$ containing $\rho$, which must therefore be equal to $\Cc$.

  Finally, to see that $\eta$ is unique, note that if $\Cc$ is the closure of a connected component of $\Xf_\eta$ then
  it has an open dense subset of points for which the restriction to $I_F^0$ is $G^\circ$-conjugate to $\eta$. This can
  clearly happen for only one $G^\circ$-conjugacy class of $\eta$.
\end{proof}

\section{The $\Sigma$-regular locus}
\label{sec:sreg}

\subsection{Regular elements and the Steinberg morphism}
\label{sec:sigma-reg-chevalley}

\emph{For this subsection only,} let $G$ be a split connected reductive group defined over a ring $\Lambda$. Let $T$
be a split maximal torus and $B \supset T$ a Borel subgroup. Let $W$ be the Weyl group of $G$ relative to $T$. Let
$r$ be the rank of $T$.

Consider the conjugation action of $G$ on itself, and denote the GIT quotient $\Spec \Lambda[G]^{G}$ by $G\git G$. We
also define 
\[T\git W = \Spec \Lambda[T]^{W} = \Spec \Lambda[X^*(T)]^W.\] 

\begin{theorem}\label{thm:chevalley}
  \begin{enumerate}
  \item The restriction morphism
    \[T\git W \to G \git G\] is an isomorphism.
  \item The formation of $T\git W$ and $G \git G$ is compatible with base change: for any $\Lambda$-algebra $\Lambda'$,
    \[G_{\Lambda'} \git G_{\Lambda'} \cong (G \git G)_{\Lambda'}\]
    and similarly for $T\git W$.
  \end{enumerate}
\end{theorem}
\begin{proof}
  \begin{enumerate}
  \item Over a field this is a theorem of \cite{steinbergRegularElementsSemisimple1965}. Over a general ring
    see \cite{leeAdjointQuotientsReductive2015}.
  \item This is \cite[Lemma~4.2]{leeAdjointQuotientsReductive2015} for $G \git G$, and
    \cite[Lemma~3.1]{leeAdjointQuotientsReductive2015} for $T \git W$.
  \end{enumerate}
\end{proof}

If $L$ is a field (and a $\Lambda$-algebra) then, for any $g \in G(L)$, its centraliser $G_g$ has dimension $\ge r$, and
we say that $g$ is regular if equality holds. Then
\[G^{\reg} = \{g \in G: \text{$g$ is regular in $G(\kappa(g))$}\}\] is an open subset of $G$ (by
semicontinuity, \cite[Proposition~4.1]{bertinGeneralitesPreschemasGroupes1965}).

\begin{theorem}\label{thm:sigma-reg} Suppose that $G_{\der}$ is simply connected. Then
  \begin{enumerate}
\item The natural morphism
  \[\Ch : G \to G\git G\]
  is faithfully flat, and its restriction to $G^{\reg}$ is smooth and surjective.
\item There is an isomorphism
  \[G\git G \cong \AA^r_\Lambda \times G/G_{\der}.\]
  \end{enumerate}
\end{theorem}
\begin{proof}
  \begin{enumerate}
  \item By the fibrewise criterion for flatness, we may check this after base-change to a field. It then follows from
    \cite[Propositions~2.4 and~2.5]{bezrukavnikovModularAffineHecke2022}.
  \item This is \cite[Lemma~2.3]{bezrukavnikovModularAffineHecke2022}  over a field; however, the proof works over
    $\Lambda$ if one takes $\chi_\alpha$ in that proof to be the character of $\Ind_B^G \omega_\alpha$ rather than its
    socle. (In the semisimple simply connected case, see also \cite[Corollary~5.2]{leeAdjointQuotientsReductive2015}).
  \end{enumerate}
\end{proof}

Finally, we recall the statement of \cite[Theorem~7.13]{cotnerCentralizersSectionsReductive2022} (noting that, if
$G_{\der}$ is simply connected, then `regular' and `strongly regular' in \cite{cotnerCentralizersSectionsReductive2022}
are equivalent).

\begin{proposition} (Cotner) \label{prop:cotner} Suppose that $\Lambda = L$ is a field, that $G_{\der}$ is simply
  connected, and that $Z_G$ is smooth. Then, for every regular element $g \in G(L)$, the centralizer $G_{g}$ is smooth.
\end{proposition}

\subsection{Application to $\Xf_G$}

\label{sec:char-sigma}

Return to the notation of section~\ref{sec:notation}. For this subsection we assume that $G$ is \emph{unramified}, in
the sense that $I_F$ acts trivially on $G^\circ$. We then may assume that $I_F \subset \ker(W_F \to \Gamma)$, and consider
only \emph{tame} parameters: those that factor through $W_F^0/P_F$. Define $\Xf^{\tm}_G$ to represent the functor on $\Lambda$-algebras
\[\Xf^{\tm}_G(R) = \{\text{$L$-parameters $\rho : W_F^\circ/P_F \to G(R)$}\}.\]
\begin{remark}\label{rmk:unram}
  By \cite{datModuliLanglandsParameters2020}~(4.4) and~(4.5), the critical case to consider is that of tame parameters
  for tamely ramified $G$ --- and when considering $\Xf^{\tm}_{G, \ZZ_l}$ it is possible to assume that the order of the
  automorphism $\sigma$ is a power of $l$. In particular, if $G$ is semisimple and simply connected, then it is no great
  loss to assume that $G$ is unramified so long as $l > 3$.
\end{remark}
We have that
\[W_F^0/P_F = \pres{\sigma, \Fr : \Fr \sigma \Fr^{-1} = \sigma^q}\]
and, for a parameter $\rho : W_F^0 \to G(R)$, we write
\[\rho(\sigma) = \Sigma_\rho \in G^\circ\]
(recalling that $G$ is unramified) and
\[\rho(\Fr) = (\Phi_\rho, \Fr).\]
Let $p_\Sigma$ and $p_\Phi$ be the morphisms $\Xf^{\tm}_G \to G^\circ$ taking $\rho$ to $\Sigma_\rho$ and $\Phi_\rho$
respectively.

Let $A_G = G^\circ\git G^\circ \cong T \git W$.  By Theorem~\ref{thm:chevalley} we have a morphism
\[\Ch: G^\circ \to A_G.\]
Note that $\Fr$ acts on $W$, $T$, and $A_G$. Let $[q] : G^\circ \to G^\circ$ be the $q$th power map.

The morphisms $\Fr$ and $[q]$ each descend to morphisms, also denoted $\Fr$ and $[q]$, from $G^\circ\git G^\circ$ to
$G^\circ \git G^\circ$. We let 
\[B_G = (G^\circ\git G^\circ)^{\Fr^{-1}[q]} = A_G^{\Fr^{-1}[q]}\] be the fixed-point
subscheme of $A_G$ under the map $\Fr^{-1}[q]$.

\begin{lemma}\label{lem:BG} The morphism $\Xf^\tm_G \to A_G$ sending $\rho$ to $\Ch(\Sigma_\rho)$ factors
  through a map
  \[\Ch_\Sigma : \Xf^\tm_G \to B_G.\]
\end{lemma}

\begin{proof} If $\rho$ is an $R$-point of $\Xf^\tm_G$ for a $\Z[1/p]$-algebra $R$, then
  $\rho(\Fr) \Sigma_\rho \rho(\Fr)^{-1} = \Sigma_\rho^q$. This is equivalent to
  \[\Phi_\rho {}^{\Fr} \Sigma_\rho \Phi_{\rho}^{-1} = \Sigma_{\rho}^q\]
  and so \[{}^{\Fr}\Ch(\Sigma_\rho) = \Ch({}^{\Fr}\Sigma_\rho) = \Ch(\Sigma_\rho^q) = [q]\Ch(\Sigma_\rho).\] Thus
  $\Ch(\Sigma_\rho) \in (G^\circ\git G^\circ)^{\Fr^{-1}[q]} = B_G$ as required.
\end{proof}

\begin{remark}\label{rem:base-change} Suppose that $\Lambda$ is a $\Z[1/p]$-algebra and
  let $A_{G, \Lambda} = T_\Lambda\git W$ and $B_{G, \Lambda} = A_{G, \Lambda}^{\Fr^{-1}[q]}$. By
  Theorem~\ref{thm:chevalley} part~(2), $A_{G, \Lambda} = A_G \times_{\Spec\Z[1/p]}\Spec \Lambda$. It follows formally that
  $B_{G, \Lambda} = B_G \times_{\Spec \Z[1/p]} \Lambda$. Finally, the map $\Ch_\Lambda : G^\circ_\Lambda \to A_{G, \Lambda}$
  provided by Theorem~\ref{thm:chevalley} part~(1) agrees with the base change of $\Ch$ to $\Lambda$, and
  Lemma~\ref{lem:BG} holds over $\Lambda$.
\end{remark}

Let $\Xf_G^{\sreg} = p_\Sigma^{-1}(G^{\circ,\reg})$. It is an open subscheme of $\Xf^\tm_G$.

\begin{theorem}\label{thm:sreg} Suppose that $G^\circ_{\der}$ is simply connected and that $\Lambda$ is a
  $\Z[1/p]$-algebra such that $Z_{G^\circ, \Lambda}$ is smooth over $\Lambda$.  Then the restriction
  \[\Ch_{\Sigma, \Lambda}|_{\Xf_{G,\Lambda}^{\sreg}} : \Xf_{G,\Lambda}^{\sreg} \to B_{G,\Lambda}\]
  is smooth and surjective.
\end{theorem}
\begin{proof}
  Let
  \begin{align*}\Zf^{reg} &= \{\Sigma \in G^{\circ, \reg} : \Sigma^q \in G^{\circ, \reg} \text{ and } \Ch({}^{\Fr}\Sigma)
                            = \Ch(\Sigma^q).\} \\
    &\subset G^{\circ, \reg} \times_{A_G} B_G
  \end{align*}
  with the latter inclusion being an open immersion. We show that in fact it is an equality. For this, it suffices to
  show that, if $\Sigma \in G^{\circ}(\kappa)$ is regular, with $\kappa$ an algebraically closed field (and
  $\Lambda$-algebra), such that $\textrm{Ch}({}^{\operatorname{Fr}} \Sigma)$ and $\textrm{Ch}(\Sigma)$ are equal, then
  ${}^{\operatorname{Fr}} \Sigma$ is conjugate to $\Sigma^q$. Considering the $su$ decomposition of $\Sigma$ and replacing
  $G$ by $Z_{G}(s)^{\circ}$, it is enough to prove that if $\Sigma$ is regular unipotent then so are
  ${}^{\operatorname{Fr}} \Sigma$ and $\Sigma^q$. The first is clear. For the second, we may assume that $G$ is semisimple
  and use the characterisation of regular unipotent elements as those whose projections onto all the root spaces for
  simple roots are nontrivial (we thank Sean Cotner for help with this).

  Then the map $\Ch : \Zf^{\reg} \to B_G$ is smooth and surjective by Theorem~\ref{thm:sigma-reg}~(1), and
    $p_\Sigma|_{\Xf^{\sreg}}$ factors through a map
    \[p_\Sigma : \Xf^{\sreg} \to \Zf^{\reg}\] that we wish to show is smooth.

    Note that there is a closed immersion
    \begin{align*}
      \iota : \Zf^{reg} & \to G^{\circ, \reg} \times_{A_G} G^{\circ,\reg}\\
      \Sigma & \mapsto ({}^{\Fr}\Sigma, \Sigma^q).
    \end{align*}
    If we let
    \[c : G^\circ \times G^{\circ, \reg} \to G^{\circ, \reg} \times_{A_G} G^{\circ, \reg}\]
    be the conjugation
    morphism $c(\gamma, g) = (g,\gamma g \gamma^{-1})$ then the pullback of $\Zf^{reg}$ along $c$ is
    \begin{align*}&\{(\gamma, g, \Sigma) \in G^\circ \times G^{\circ, \reg} \times G^{\circ, \reg} : g = {}^{\Fr} \Sigma,
                    \gamma g \gamma^{-1} = \Sigma^q\} \\
                  &= \{(\gamma, \Sigma) : \gamma {}^{\Fr}\Sigma \gamma^{-1} = \Sigma^q\} \\
                  &= \Xf_G^{\sreg}.
    \end{align*}
    It therefore suffices to show that $c$ is smooth and surjective after base change to $\Lambda$. We think of $c$ as a morphism of
    $A_G$-schemes where $G^\circ \times G^{\circ,\reg}$ is regarded as an $A_G$-scheme via $\Ch$ on the second factor. Since
    $G \times G^{\circ,\reg}$ is flat over $A_G$ by Theorem~\ref{thm:sigma-reg}~(1), we may apply the fibrewise
    criterion. It is therefore enough to show that, for every geometric point $s : \Spec \kappa \to A_{G,\Lambda}$,
    \[c_s : G^\circ_\kappa \times_{\Spec \Lambda} G^{\circ,\reg}_s \to G^{\circ, \reg}_s \times_{\Spec \Lambda}
      G^{\circ, \reg}_s\] is smooth and surjective.

    By Theorem~\ref{thm:sigma-reg}~(1), $G^{\circ, \reg}_s$ is smooth and by \cite{bezrukavnikovModularAffineHecke2022}
    Proposition~2.5 it is a single $G^\circ_\kappa$-orbit. The fibre of $c_s$ above a point
    $(g, g') \in G^{\circ, \reg}_s(\kappa) \times G^{\circ, \reg}_s(\kappa)$ is then a $G^\circ_{\kappa,g}$-torsor. It
    follows from the miracle flatness theorem that $c_s$ is flat, and from Proposition~\ref{prop:cotner} that it is
    smooth.
\end{proof}

As an application, we generalise \cite{helmCurtisHomomorphismsIntegral2020} Proposition~5.3.

\begin{corollary} Suppose that the hypotheses of Theorem~\ref{thm:sreg} hold with $\Lambda = \Z_l$. Let
  $S = \Oc(\Xf^\tm_{G, \Lambda})^{G^\circ_\Lambda}$ and $R = \Oc(B_{G,\Lambda})$. Then the natural map $\Ch_\Sigma^* : R \to S$ is injective
    with saturated image.
\end{corollary}
\begin{proof}
  Injectivity follows from the faithful flatness of $\Ch_\Sigma$ restricted to the regular locus. If $l \ne p$ is a
  prime, then the same argument shows that $\Ch_{\Sigma}^* \otimes \F_l = \Ch_{\Sigma, \F_l}^*$ (see
  Remark~\ref{rem:base-change}) is injective. Since $S$ is a torsion-free $\Lambda$-module, this
  implies that the image of $\Ch^*_{\Sigma}$ is saturated.
\end{proof}

\begin{remark}\label{rmk:ram} Remark~\ref{rmk:unram} notwithstanding, it would be desirable to generalise these results
  to the case that $G$ is tamely ramified. In this case $G^\circ \git G^\circ$ should be replaced by
  \[A_{G, \sigma} = (G^\circ \rtimes \sigma) \git G^\circ = G^\circ \git_\sigma G^\circ\] (where $\git_\sigma$ denotes the quotient for
  $G^\circ$ acting by $\sigma$-twisted conjugation). The $q$-power map and Frobenius map define $G^\circ$-equivariant morphisms
  \[G^\circ \rtimes \sigma \to G^\circ \rtimes \sigma^q\] and we let $B_{G,\sigma} = A_{G,\sigma}^{\Fr^{-1}[q]}$ be
  their equaliser. There is a notion of $\sigma$-regularity, for which see \cite{xiaoVectorvaluedTwistedConjugate2019}
  section~5. We then expect that, under appropriate hypotheses on $G^\circ$ (perhaps those of
  \cite{xiaoVectorvaluedTwistedConjugate2019} Remark~4.3.4 together with a hypothesis on the centre) there will be a
  smooth morphism $\Xf_G^{\sreg} \to B_{G,\sigma}$. For this we would require generalizations of the results of
  section~\ref{sec:sigma-reg-chevalley}: that is, the results of \cite{xiaoVectorvaluedTwistedConjugate2019} sections~4
  and~5 over a general base and for a reductive (rather than semisimple) group, and Proposition~\ref{prop:cotner} in the
  twisted context.
\end{remark}

\begin{remark} If $G$ is the $L$-group of an unramified group then we have a representation-theoretic interpretation of
  $B_G$, see Remark~\ref{rmk:endomorphism}.  It would be interesting to have a similar representation-theoretic
  interpretation of $B_{G,\sigma}$, as defined in Remark~\ref{rmk:ram}, in the tamely ramified case.
\end{remark}

\section{Levi subgroups}
\label{sec:levi}

For this section we no longer assume that $G$ is unramified. Let $M \subset G$ be a standard Levi subgroup in the
sense of \cite[section 3]{borelAutomorphicFunctions1979}: that is, a subgroup of the form $M^\circ \rtimes \Gamma$ where
$M^\circ$ is a $\Gamma$-stable standard (with respect to our chosen $T$) Levi subgroup of $G^\circ$. Since $\Gamma$
preserves a Borel subgroup $B \supset T$, there is a $\Gamma$-stable standard parabolic subgroup $P^\circ$ with Levi
$M^\circ$, and we write $P = P^\circ \rtimes \Gamma$. Let $U$ be the unipotent radical of $P^\circ$ and $U^{-}$ its
opposite. We have a decomposition $\Lie(G) = \Lie(M) \oplus \Lie(U) \oplus \Lie(U^{-})$ that is stable under the action
of $W_F^0$ on $\Lie(G)$. We also have a natural closed immersion
\[\Xf_M \into \Xf_G\]
and a conjugation map
\begin{align*}
  c : G^\circ \times \Xf_M & \to \Xf_G \\
  (g, \rho) &\mapsto g \rho g^{-1}.
\end{align*}

We will show that there is an open and fibrewise dense subset $V \subset \Xf_M$ such that $c : G^\circ \times V \to \Xf_G$ is smooth. If we
fix a separated filtration $(P^d_F)_{d \ge 0}$ of $P_F$ by open subgroups then
\[\Xf_M^{d} = \{\rho \in \Xf_M : \rho|_{P_F^d} \text{ is trivial}\}\]
is a union of connected components of $\Xf_M$. 

Let $d \ge 1$ and let $r$ be an integer, depending on $d$, such that:
\begin{itemize}
\item $\Fr^r = e$ in $\Gamma$;
\item The conjugation action of $\Fr^r$ on $P_F/P_F^d$ is trivial; and
\item For any semisimple $s \in G^\circ(\bar{\Q})\rtimes\sigma$ such that $s$ is conjugate to $s^q$, we have $s = s^{q^r}$.
\end{itemize}
Such an integer must exist; for instance, choose a faithful representation $G \into GL_N$ and take
$r = n!\ord_{\Gamma}(\Fr)|\Aut(P_F/P_F^d)|$. The first two conditions immediately follow from the divisibility of $r$ by
its second and third factors, respectively. The third condition follows from the divisibility of $r$ by $n!$ and the fact that, if $s \in GL_N(\bar{\Q})$ with
$s$ conjugate to $s^q$, then $s = s^{q^{n!}}$. See also \cite{datModuliLanglandsParameters2020} Lemma~2.2 for a similar argument. Say
that $m \in M^\circ(R)\rtimes\Fr$ is \emph{avoidant at depth $d$} if:
\begin{itemize}
\item $\ad_{m} - 1$ and $\ad_{m} - q$ are invertible on $\Lie(U)(R)$ and $\Lie(U^{-})(R)$;
\item if $\chi$ is the characteristic polynomial of $\ad_{m^r}$ on $\Lie(M)(R) \oplus \Lie(U)(R)$, and
  $\gamma \in \Gamma$, then $\chi(\ad_{{}^\gamma m^r})$ is invertible on $\Lie(U^{-})(R)$; and
\item the same condition holds with the roles of $U$ and $U^{-}$ reversed.
\end{itemize}
The relevance of the first condition will become clear, while the second and third are used via the following lemma.
\begin{lemma}\label{lem:ab-conjugate} Suppose that $a, b \in M^\circ(\bar{\Q})$ satisfy the condition that, if $\chi$ is
  the characteristic
    polynomial of $\ad_{a}$ on $\Lie P_{\bar{\Q}} = \Lie M_{\bar{\Q}} \oplus \Lie U_{\bar{\Q}}$, then $\chi(\ad_b)$ is
    invertible on $\Lie(U^{-}_{\bar{\Q}})$, and similarly with the roles of $U$ and $U^{-}$ reversed.

    If $b = gag^{-1}$ for some $g$ in $G^\circ(\bar{\Q})$, then $g \in M^\circ(\bar{\Q})$.
  \end{lemma}
  \begin{proof}
    Let $g \in G^\circ(\bar{\Q})$ such that $gag^{-1} = b$. Then $\ad_g$ takes each (generalised) eigenspace of $\ad_a$
    to a generalised eigenspace of $\ad_b$ with the same eigenvalue. The assumption on $a$ and $b$ then implies that
    $\Lie(M_{\bar{\Q}}) \oplus \Lie(U_{\bar{\Q}})$ is stable under $\ad_g$, from which it follows that the parabolic
    subgroup $M^\circ_{\bar{\Q}} U_{\bar{\Q}}$ is normalised by $g$. Since a parabolic subgroup is its own normaliser,
    we see that $g \in M^\circ U$. Arguing similarly with $U$ and $U^{-}$ reversed, we see that
    $g \in M^\circ U \cap M^{\circ}U^{-} = M^\circ$ as required.
  \end{proof}

  Write $M^{d,a}$ for the open subscheme of $M^\circ\rtimes\Fr$ consisting of those $\Phi$ avoidant at depth $d$.  Let
  $\Xf_M^{d,a} = \{\rho \in \Xf^{d}_M : \rho(\Fr) \in M^{d,a}\}$, and let $\Xf_M^a = \bigcup_{d \ge 1} \Xf_M^{d,a}$.  We have a
  morphism
  \[c : G^\circ \times \Xf_M \to \Xf_G\] given by $c(g, \rho) = g\rho g^{-1}$. This descends to a morphism
  $\bar{c} : G^\circ \times^{M^\circ} \Xf_M \to \Xf_G$.
\begin{theorem}\label{thm:levi}
\begin{enumerate}
\item The open subscheme $\Xf_M^{a} \subset \Xf_M$ is fibrewise dense.
\item The restriction of $c$ to $G^\circ \times \Xf_M^{a}$ is smooth.
\item The restriction of $\bar{c}$ to $G^\circ \times^{M^\circ} \Xf_M^{a}$ is \'{e}tale.
\end{enumerate}
\end{theorem}
\begin{proof}
  If $M = G$ then all of these statements are trivial. So assume that $M$ is a \emph{proper} standard Levi subgroup. It is
  enough to prove the theorem for $\Xf_M^{d,a}$ for each fixed $d \ge 1$, so fix such a $d$.
  \begin{enumerate}
  \item Let $J \subsetneq \Delta$ be the set of simple roots associated to $P$. As in the proof of
    \cite{borelAutomorphicFunctions1979} Lemma~3.5, if we let $A$ be the maximal subtorus of $Z(M^\circ)$ and
    \[S = \{s \in A \cap G^{\circ}_{\der}: \alpha(s) = \beta(s) \text{ for all
        $\alpha, \beta \in \Delta \setminus J$}\},\] then $S$ is a rank one torus and $M = Z_G(S)$. If
    $\lambda \in \Delta \setminus J$, then the weights of $S$ acting adjointly on $\Lie(U)$ (respectively,
    $\Lie(U^{-})$) are positive (respectively, negative) powers of $\lambda|_S$, which is independent of the choice of
    $\lambda$ by definition. Since $W^0_F$ acts on $S \cong \GG_m$ by automorphisms and also preserves $U$ and
    $\lambda|_{S}$, it must fix $S$. For $s \in S$, let $\chi_s$ be the unramified character of $W_F^0$ sending $\Fr$ to
    $s$.

    Suppose that $x \in \Spec \Z[1/p]$ with residue field $\kappa$ and that $\rho$ is a closed point of the fibre
    $\Xf^d_{M,\kappa}$. Consider the map
    \begin{align*}S_\kappa & \to \Xf_{M, \kappa} \\
      s &\mapsto \chi_s \rho.
    \end{align*}
    The eigenvalues of $\ad_{s\rho(\Fr)}$ acting on $\Lie(U)$ are of the form $\lambda(s)^n\mu$ for $\mu$ an eigenvalue
    of $\ad_{\rho(\Fr)}$ and $n$ positive. Similarly, for each $\gamma \in \Gamma$, the eigenvalues of
    $\ad_{s {}^\gamma\rho(\Fr)}$ on $\Lie(U^{-})$ are of the form $\lambda(s)^n \mu'$ for $\mu'$ an eigenvalue of
    $\ad_{{}^\gamma\rho(\Fr)}$ and $n$ \emph{negative}. From this and the same considerations with the roles of $U$ and
    $U^{-}$ reversed, it follows that $s\rho(\Fr)$ is avoidant at depth $d$ for all but finitely many $s \in S_\kappa$. In
    particular, $\rho$ is in the closure of $\Xf_{M,\kappa}^{d,a}$ as required.
  \item Let $Z$ be the locally closed subscheme of $\Xf^{d}_G$ given by
    \[Z= \{\rho \in \Xf^{d}_G : \rho(\Fr) \in M^{d,a}\}.\]
    \begin{enumerate}
    \item We first show that the morphism $G^\circ \times Z \to \Xf^d_G$ sending $(\gamma, \rho)$ to
      $\gamma\rho\gamma^{-1}$ is smooth. This map is the pullback of the conjugation map
      $c : G^\circ \times M^{d,a} \to G$ along the projection $\Xf^d_G \to G$ sending $\rho$ to $\rho(\Fr)$. It is
      enough to show that $c$ is smooth. Since $G^\circ \times M^{d,a}$ and $G$ are smooth over $\Z[1/p]$, it is enough
      (by the fibrewise criterion of smoothness, \cite[Proposition 17.8.2]{EGAIV-4}, together with \cite[Proposition
      17.7.1 and Théorème 17.11.1(d)]{EGAIV-4}) to show that the derivative $Dc$ is surjective on geometric tangent spaces. Since
      the map $c$ is equivariant for the action of $G^\circ$ we may check this at a point
      $(e, m) \in (G^\circ \times M^{d,a})(\bar{\kappa})$ for an algebraically closed field $\kappa$. Identify the
      tangent spaces of $G^\circ_\kappa \times M^{d,a}_\kappa$ (resp.\ $G_\kappa$) at $(e,m)$ (resp.\ $m$) with
      $\Lie G_\kappa \oplus \Lie M_\kappa$ (resp.\ $\Lie G_\kappa$) via left multiplication by $(e,m)$ (resp.\ $m$). A
      computation shows that for $(X, Y) \in \Lie G_\kappa \oplus \Lie M_\kappa$
      \[Dc(X,Y) = (\ad_{m}^{-1} - 1)X + Y,\]
      so that $Dc$ is surjective by the assumption that $m$ is avoidant at depth $d$.
    \item We must show that the closed immersion $\iota : \Xf_M^{d,a} \to Z$ is an isomorphism. Since
      $\Xf_G$ is $\Z[1/p]$-flat and reduced, $Z$ also has these properties. Since $\iota$ is a closed immersion of
      $\Z[1/p]$-flat schemes, it is enough to show that $\iota \otimes_{\Z[1/p]} \Q$ is an isomorphism.
    \item Since $Z \otimes_{\Z[1/p]} \Q$ is reduced and of finite type over $\Q$, it is enough to show that
      $\iota\otimes \bar{\Q}$ is a
      bijection on closed points. Thus we must show that, if $\rho \in Z(\bar{\Q})$, then $\rho \in
      \Xf^d_M(\bar{\Q})$. In other words, if $\rho$ is an $L$-parameter with $\rho(\Fr) \in M^{d,a}(\bar{\Q})$, then
      we also have $\rho(\sigma) \in M(\bar{\Q})$ and $\rho(g) \in M(\bar{\Q})$ for all $g \in P_F/P_F^d$.

      Let $\rho$ be such a parameter, and let $\rho(\sigma) = (\Sigma_s, \sigma) \Sigma_u$ be the Jordan decomposition
      of $\rho(\sigma)$ (we must have that $\rho(\sigma)_u \in G^\circ(\bar{\Q})$ since $\bar{\Q}$ has characteristic
      zero). 
    \item We show that $\Sigma_u \in M^\circ(\bar{\Q})$. Write $\Sigma_u = \exp(N)$ for $N \in \Lie(G)(\bar{\Q})$. Then (by
      uniqueness of Jordan decomposition) we have $\ad_{\rho(\Fr)}(N) = qN$. By our assumption that $\rho(\Fr)$ is
      avoidant at depth $d$, we have
      $N \in \Lie(M)(\bar{\Q})$ and so $\Sigma_u \in M^\circ(\bar{\Q})$.
    \item We next show that $\Sigma_s \in M^\circ(\bar{\Q})$. Let $r$ be as above. Then
      \[\rho(\Fr)^r(\Sigma_s, \sigma)\rho(\Fr)^{-r} = (\Sigma_s, \sigma)^{q^r} = (\Sigma_s,\sigma),\]
      which may be rewritten as
      \[(\Sigma_s)^{-1}\rho(\Fr)^r \Sigma_s = {}^\sigma \rho(\Fr)^r.\]
      Applying Lemma~\ref{lem:ab-conjugate}, we see that $\Sigma_s \in M^\circ(\bar{\Q})$ as required.
    \item Finally, we show that $\rho(g)\in M(\bar{\Q})$ for $g \in P_F/P_F^d$. This is identical to the previous step:
      if $\pi = \rho^\circ(g)$ we obtain \[\pi^{-1} \rho(\Fr)^r \pi = {}^{g}\rho(\Fr)^r\]
      and therefore that $\pi \in M^\circ(\bar{\Q})$.
    \end{enumerate}

  \item It follows from part 2 that $\bar{c}$ is smooth of relative dimension
    \[ \dim (G^\circ \times^{M^\circ} \Xf_M^a) - \dim \Xf_G = \dim G^\circ - \dim M^\circ + \dim \Xf_M -
      \dim \Xf_G = 0,\]
    by Theorem~\ref{thm:dhkm}, and hence \'{e}tale.
  \end{enumerate}
\end{proof}
\begin{remark}
  In fact, $\bar{c}$ descends further to a map $G^\circ \times^{N_{G^{\circ}}(M^\circ)} \Xf_M^a \to \Xf_G$ which one can
    show is an open immersion. This implies that $\bar{c}$ is a torsor under
    \[N_{G^\circ}(M^\circ)/M^\circ \cong N_W(W_{M^\circ})/W_{M^\circ}
    \]
    (see \cite{malleLinearAlgebraicGroups2011} Corollary 12.11 for this isomorphism).
  \end{remark}
  \begin{remark}\label{rem:base-change-2}
    The formation of $\Xf_M^a$ is compatible with base change, as is the property of fibrewise density, and so we obtain
    Theorem~\ref{thm:levi} over any $\Z[1/p]$-algebra $\Lambda$.
    \end{remark}
\section{(Non-)Examples}
  \label{sec:examples}
  In all the following examples we consider unramified $G$ and tamely ramified $L$-homomorphisms only. If $\kappa$ is an
  algebraically closed field of characteristic distinct from $p$, then the irreducible components of $\Xf_{G,\kappa}$ are in
  bijection with the $G^\circ(\kappa)$-conjugacy classes of pairs $(\Sigma, C)$, where $\Sigma \in G^\circ(\kappa)$ and $C$
  is an element of
  \[\pi_0(\{\Phi \in G^\circ_\kappa : \Phi {}^{\Fr}\Sigma\Phi^{-1} = \Sigma^q\}).\]

  If we fix $\Sigma$ and fix $\Phi_0$ with $\Phi_0\Sigma\Phi_0^{-1} = \Sigma^q$, then by Corollary~\ref{cor:irr-cpts} the equivalence classes of pairs
  $(\Sigma, C)$ are in bijection with the $\Ad_{\Phi_0}$-twisted conjugacy classes in $\pi_0(C_{G^\circ}(\Sigma))$.  We call
  the corresponding components the components of type $\Sigma$.

  Let $\Lambda = W(\bar{\F}_l)$; in the rest of this section, we consider everything over $\Lambda$. Our main interest is in
  describing the local geometry of $\Xf_{G}$ around generic points of components of the special fibre. By
  \cite{datModuliLanglandsParameters2020} section~4.2, it is possible to reduce the general case to that of tame
  parameters such that the semisimple part of $\Sigma$ has $l$-power order --- and so, modulo $l$, $\Sigma$ is
  unipotent. We therefore focus on such components in the discussion below. We call the irreducible components of
  $\Xf^\tm_{G,\bar{\F}_l}$ on which $\Sigma$ is unipotent the \emph{unipotent components}. We have the morphism
  $\Ch : \Xf_G^\tm \to (T \git W)^{\Fr^{-1}[q]}$ and the unipotent components are those in the preimage of the identity
  $\bar{e} \in (T\git W)^{\Fr^{-1}[q]}(\bar{\F}_l)$.

\subsection{$GL_n$}

Since centralizers are connected in $G = GL_n$, the irreducible unipotent components of $\Xf_{G, \bar{\F}_l}$ correspond
to unipotent conjugacy classes in $GL_n$, which are parametrized by partitions of $n$ giving the Jordan canonical form
of $\Sigma$. For every conjugacy class $[\Sigma]$, we can find a representative such that $\Sigma$ is regular in a
standard Levi subgroup $M$ (the one corresponding to the same partition of $n$). By Theorems~\ref{thm:levi}
and~\ref{thm:sigma-reg}, we have open subsets $U \subset \Xf_{M}$ and $V \subset \Xf_{G}$, each containing the generic
point of the component corresponding to $\Sigma$, and smooth morphisms
\[G \times U \to V\] 
and
\[U \to (T\git W_M)^{[q]}.\] In particular, the completed local ring of $\Xf_G$ at a point in $V(\bar{\FF}_l)$ will be
formally smooth over $(T\git W_M)_{\bar{e}}^{[q]}$. We therefore recover the geometric results of
\cite{shottonGenericLocalDeformation2022} in this case (and there is no need to assume that $\Sigma$ is unipotent
here).

  \subsection{$SL_2$}
\label{sec:sl2}
Suppose that $G = SL_2$ (with $\Gamma$ trivial). There are two conjugacy classes of unipotent $\Sigma$.

If $\Sigma = \twomat{1}{1}{0}{1}$, then we may take $\Phi_0 = \twomat{q}{0}{0}{1}$. We have
$\pi_0(C_{G^\circ}(\Sigma)) = \pi_0(Z(SL_2)) = \pi_0(\mu_2)$, and so there are two components of $\Xf_{\bar{\FF}_l}$ of type $\Sigma$
unless $l = 2$, in which case there is only one. If $l \ne 2$, then there is an open subset $U$ of $\Xf_G$ intersecting
each of these components such that the map
\[U \to (T\git W)^{[q]} = \Spec \ZZ_l[x + x^{-1}]/((x^q + x^{-q} - (x + x^{-1}))]\]
is smooth.
  
If $\Sigma = \twomat{1}{0}{0}{1}$, then there is only one component $C_{nr}$ of type $\Sigma$. There is an open subset
$V$ of $\Xf^\tm_{T} \cong \mu_{q-1} \times \GG_m$ such that the morphism $G \times V \to \Xf_G$ is smooth and
has image intersecting $C_{nr}$ (necessarily in an open subset).

\subsection{Unitary groups}

We take $G$ to be the $L$-group of an unramified quasisplit unitary group over $F$. Thus,
\[G = G^\circ \rtimes \Gal(E/F)\]
where $G^\circ = GL_n$, the extension $E/F$ is unramified quadratic, and the nontrivial element $c \in \Gal(E/F)$ acts on $G^\circ$ as
\[{}^cg = Jg^{-T}J^{-1}\] for $J$ an antidiagonal matrix of alternating $1$s and $-1$s. Then $c$ preserves the standard
maximal torus $T$ and Borel $B$ in $G^\circ$.

Since centralizers in $G^\circ$ are connected, the irreducible components of $\Xf_{G,\bar{\F}_l}$ are in bijection with
$G^\circ(\bar{\F}_l)$-conjugacy classes of $\Sigma \in G^\circ(\bar{\F}_l)$ such that ${}^c\Sigma \sim \Sigma^q$.  In
particular, the unipotent irreducible components are in bijection with the unipotent conjugacy classes in $G^\circ(\bar{\F}_l) \cong
GL_n(\bar{\F}_l)$, since ${}^c \Sigma \sim \Sigma$ for $\Sigma$ unipotent.

Note that $\Fr$ acts on $T\git W$ in the same way as the map induced by $t \mapsto t^{-1}$. We therefore have a formally
smooth morphism
$\Xf_G^{\reg} \to (T\git W)^{[-q]}$, by Theorem~\ref{thm:sigma-reg}.

However, it is no longer true that every unipotent conjugacy class of $G^\circ$ contains an element regular in a
standard Levi subgroup of $G$. Indeed, such Levi subgroups must be preserved by $c$, and so we only get partitions of
$n$ for which all except at most one part occurs an even number of times.

For example, when $n = 3$, the unipotent components of $\Xf_G$ to which Theorem~\ref{thm:levi} applies are those
corresponding to the partitions $(1,1,1)$ and $(3)$; we do not obtain a description of the local geometry of $\Xf_G$
about points for which $\Sigma$ is unipotent of type $(2,1)$.

\subsection{$GSp_4$}
\label{sec:gsp4}
Let $G = \GSp_4$, with $l > 2$. We take the symplectic form to be that given by the matrix
\[J = \begin{pmatrix} & & & 1 \\ & & 1 & \\ & -1 & & \\ -1 & & &\end{pmatrix}.\] We describe the unipotent components of
$\Xf_{\bar{\FF}_l}$. The conjugacy class of a unipotent matrix $u \in \GSp_4(\bar{\F}_l)$ is determined by the rank of $u-1$. The
centralizer of $u$ is connected unless the rank of $u-1$ is two, in which case the centralizer has two connected
components (see the table on p400 of \cite{carterFiniteGroupsLie1985}). We label the unipotent components by the rank of $\Sigma - I$ at
a generic point, so that they are $C_0$, $C_1$, $C_{2A}$, $C_{2B}$, and $C_3$.
  
Firstly, on the unramified component $C_{0}$ we have $\Sigma = I$. We may apply Theorem~\ref{thm:levi} to this component,
with the Levi subgroup being the standard maximal torus, to see that there is an \'{e}tale neighbourhood $U$ of the
generic point of $C_0$ such that $U$ is smooth over $T^{[q]} \cong \mu_{q-1}^3$.

Secondly, we have the component $C_1$ on which $\Sigma - I$ has rank $1$ (generically). There is a unique such component
since the centralizer of any $\Sigma$ such that $\Sigma - I$ has rank 1 is connected. Letting $M \cong GL_2\times \GG_m$ be the Levi
subgroup of block diagonal matrices with block sizes 1, 2, 1, there is an \'{e}tale neighbourhood $U$ of the generic
point of $C_1$ such that $U$ is smooth over $(T\git W_M)^{[q]}$. 

Thirdly, we have the component $C_3$ on which $\Sigma - I$ has rank $3$ (generically). There is an open neighbourhood
$U$ of the generic point of $C_3$ that is smooth over $(T\git W)^{[q]}$. 

Lastly, consider
  \[u = \begin{pmatrix}
  1 &1  & 0 & 0 \\ 0 & 1 & 0 & 0 \\ 0 & 0 & 1 & -1 \\ 0 & 0 & 0 & 1
\end{pmatrix}.\] Then $C_G(u)$ has two components. The Frobenius-twisted conjugacy classes of
$\pi_0(C_G(u)) \cong \Z/2$ are in this case just the conjugacy classes and there are therefore two components of
$\Xf_u$. One of them, which we call $C_{2A}$, contains the points with $\Sigma = u$ and
  \[\rho(\Fr) = \begin{pmatrix}
  \lambda q &0  & 0 & 0 \\ 0 & \lambda & 0 & 0 \\ 0 & 0 & q & 0 \\ 0 & 0 & 0 & 1
\end{pmatrix},\]
 while the other, which we call $C_{2B}$, contains the points with $\Sigma = u$ and
\[\rho(\Fr) = \begin{pmatrix}
    0 &0 & -\lambda q & 0 \\ 0 & 0 & 0 & \lambda \\ -q & 0 & 0 & 0 \\ 0 & 1 & 0 & 0
  \end{pmatrix}.\]

The component $C_{2A}$ may be understood by applying Theorem~\ref{thm:levi} with the Levi subgroup $M \cong
GL_2\times_{\GG_m}GL_2$ consisting
of block diagonal matrices with block sizes (2, 2), and we see that an \'{e}tale neighourhood $U$ of the generic point of
$C_{2A}$ will be smooth over $(T\git W_M)^{[q]}_1$.  The component $C_{2B}$, however, falls outside the scope of
our results.

\begin{remark}
  We can show by explicit calculation that, if $q \equiv -1 \bmod l$, then any open neighbourhood of the generic point
  of $C_{2B}$ intersects the components of the generic fibre of type $C_{2B}$ \emph{and} $C_3$, so that no statement
  like~\ref{thm:levi} can be true.

  If $q \equiv 1 \bmod l$ and $l \ge 5$, then one can actually prove a statement like Theorem~\ref{thm:levi} for the
  component $C_{2B}$ with $M$ taken
  to be the pseudo-Levi subgroup (in the sense of \cite{mcninchComponentGroupsUnipotent2003}) 
  \begin{align*}M &= Z_{\GSp_4}(\Phi_0) \\
    \intertext{where}
    \Phi_0 &= \begin{pmatrix}
  1 &0  & 0 & 0 \\ 0 & -1 & 0 & 0 \\ 0 & 0 & 1 & 0 \\ 0 & 0 & 0 & -1
\end{pmatrix}.\end{align*}
For this, we choose $\lambda = 1$ and make a change of basis such that $\rho(\Fr)= \Phi_0$ and
  \[\Sigma = \begin{pmatrix}
  1 &0  & 1 & 0 \\ 0 & 1 & 0 & 1 \\ 0 & 0 & 1 & 0 \\ 0 & 0 & 0 & 1
\end{pmatrix} \in M.\]
In this case, we can show that an \'{e}tale neighbourhood of $C_{2B}$ is smooth over $(T\git W_M)^{[q]}$. Is there a
more general picture involving pseudo-Levi subgroups when $q \equiv 1 \mod l$ (and $l$ is sufficiently large)?
\end{remark}
\subsection{Distinguished unipotent elements}
\label{sec:gsp6}

The Bala--Carter theorem (see \cite{premetNilpotentOrbitsGood2003}) says that, in good characteristic, every
unipotent element of a reductive group is a \emph{distinguished} unipotent element of a Levi subgroup. For $G = GL_n$,
the notions of distinguished and regular coincide, but in general they do not.

Concretely, suppose that $G = \GSp_6$, $l > 2$, and that
\[\Sigma = \begin{pmatrix} 1 & 1 &  & & & \\
    & 1 & & & 1 & \\
    & & 1 & 1 & &  \\
    & & & 1 & & \\
    & & & & 1 & -1 \\ & & & & & 1\end{pmatrix}.\] Then there is a single component of type $\Sigma$. This $\Sigma$ is
distinguished unipotent, so is contained in no proper Levi subgroup of $G$, but is not regular.\footnote{Indeed, it has Jordan blocks of shape
$(4,2)$ but no proper Levi subgroup of $G$ contains elements of this shape.}  This component therefore
falls outside the scope of our results, and we do not understand the deformation rings at generic points of these
components (except for \emph{minimal} deformations as in \cite{booherMinimallyRamifiedDeformations2019}).

\begin{remark}
  If $q \equiv 1 \mod l$ then the trivial representation lies on the component of type $\Sigma$. In particular, to apply
  the Ihara avoidance method of \cite{Taylor2008-AutomorphyII} to $\GSp_{2n}$, $n > 2$, one would need to understand
  something about deformations of representations of type $\Sigma$. For $\GSp_4$, where Ihara avoidance was worked out
  in section 7.4.5 of \cite{boxerAbelianSurfacesTotally2021}, the issue does not arise, as $C_{2B}$ does not contain the
  trivial representation.
\end{remark}

\bibliography{references.bib}{}
\bibliographystyle{amsalpha}
\end{document}